\documentclass[12pt]{amsart}
\usepackage{pdfsync}
\usepackage{listings}
\usepackage{pdflscape}
\usepackage{rotating}
\usepackage{multirow}
\usepackage[utf8]{inputenc}
\usepackage[T1]{fontenc}
\usepackage{amsthm,amsmath,amssymb}
\usepackage{stmaryrd}
\usepackage[backend=bibtex,style=numeric]{biblatex}
\bibliography{eigene,andere,arxiv,mathscinet}
\usepackage{cleveref}
 \newtheorem{Thm}{Theorem}[section]
 \newtheorem{Lem}[Thm]{Lemma}
 \newtheorem{Prop}[Thm]{Proposition}
 
\theoremstyle{remark}
 \newtheorem{Expl}[Thm]{Example}

 \newtheorem{Rem}[Thm]{Remark}
\theoremstyle{definition}
 
\numberwithin{equation}{section}
\newcommand\N{\mathbb N}
\newcommand\CC{\mathbb C}
\newcommand\ol[1]{\overline{#1}}

\newcommand\disjun{\bigsqcup}
\newcommand\adhit{\triangleright}
\newcommand\twisthit{\tilde\triangleright}
\newcommand\inv{^{-1}}
\def\HM#1.#2.#3.#4.{{^{#1}_{#3}\mathcal M^{#2}_{#4}}}
\newcommand\lYD[1]{{^{#1}_{#1}\mathcal{YD}}}

\newcommand\ot{\otimes}
\newcommand{\ou}[1]{\underset{{#1}}{\otimes}}

\newcommand\Res{\operatorname{Res}}
\newcommand\Ind{\operatorname{Ind}}
\newcommand\Stab{\operatorname{Stab}}
\newcommand\Rep{\operatorname{Rep}}
\newcommand\Irr{\operatorname{Irr}}
\newcommand\Vect{\operatorname{Vect}}
\newcommand\Tr{\operatorname{Tr}}
\newcommand\lquot{\backslash}

\newcommand\Hom{\operatorname{Hom}}

\newcommand\Indtobim{\mathcal F}
\newcommand\IndtoYD{\mathcal G}
\newcommand\underlie{\mathcal U}

\newcommand\C{\mathcal C}
\newcommand\CTR{\mathcal Z}
\newcommand\RCTR{\ol{\mathcal Z}}
\newcommand\RADJ{\ol{\mathcal K}}
\begin{document}
\title[Frobenius-Schur indicators for group inclusions]{Computing Higher Frobenius-Schur Indicators in Fusion Categories Constructed from Inclusions of Finite Groups}
\author{Peter Schauenburg}
\address{Institut de Math{\'e}matiques de Bourgogne --- UMR 5584 CNRS\\
Universit{\'e} de Bourgogne\\
Facult{\'e} des Sciences Mirande\\
9 avenue Alain Savary\\
BP 47870 21078 Dijon Cedex\\
France
}
\email{peter.schauenburg@u-bourgogne.fr}
\subjclass[2010]{18D10,16T05,20C15}
\keywords{Fusion category, Frobenius-Schur indicator}
\thanks{Research partially supported through a FABER Grant by the \emph{Conseil régional de Bourgogne}}
\begin{abstract}
We consider a subclass of the class of group-theoretical fusion categories: To every finite group $G$ and subgroup $H$ one can associate the category of $G$-graded vector spaces with a two-sided $H$-action compatible with the grading. We derive a formula that computes higher Frobenius-Schur indicators for the objects in such a category using the combinatorics and representation theory of the groups involved in their construction. We calculate some explicit examples for inclusions of symmetric groups.
\end{abstract}
\maketitle

\section{Introduction}
\label{sec:introduction}

Higher Frobenius-Schur indicators are invariants of an object in a pivotal fusion category (and hence also invariants of that category). They generalize, to higher degrees and more general objects, the degree two Frobenius-Schur indicator defined for a representation of a finite group by its namesakes in 1906. Categorical versions of degree two indicators were studied by Bantay \cite{Ban:FSICFT} and Fuchs-Ganchev-Szlachányi-Vescernyés \cite{FucGanSzlVes:FSIRMCC}, indicators for modules over semisimple Hopf algebras were introduced by Linchenko-Montgomery \cite{LinMon:FSTHA} and studied in depth by Kashina-Sommerhäuser-Zhu \cite{KasSomZhu:HFSI}. The degree two indicators for modules over semisimple quasi-Hopf algebras were treated by Mason-Ng \cite{MasNg:CIFSISQHA}. The higher indicators for pivotal fusion categories that we deal with in the present paper were introduced in \cite{NgSch:CIHISQHA,NgSch:HFSIPC,NgSch:FSIESC}.

Frobenius-Schur indicators have become a tool for the structure theory and classification of fusion categories. The problem we deal with here, however, is simply how to calculate them in very specific examples. More concretely we will deal with a specific class of group-theoretical fusion categories \cite{MR1976233,MR2183279}. Degree two indicators for Hopf algebras associated with such categories have been studied in \cite{MR1919158,MR2471448}. In \cite{KasSomZhu:HFSI} formulas for higher indicators of smash product Hopf algebras associated to a group acting by automorphisms on another group were given. This class of examples includes the Drinfeld double of a finite group. For such doubles, the explicit formulas were used to study the question of integrality of the indicators in \cite{2012arXiv1208.4153I}. Extensive computer calculations, in particular with a view on the question whether the indicators of the doubles of symmetric groups are positive, were conducted in \cite{MR3103664}, examples for certain other groups can be found in \cite{MR3178056,MR2925444}.

Natale \cite{Nat:FSICFC} has derived formulas for the degree two Frobenius-Schur indicators of the objects in general group-theoretical fusion categories. Her approach is based on the fact that a group-theoretical fusion category can be written as the module category over a quasi-Hopf algebra which is known explicitly. Then the explicit definition of degree two indicators of modules over quasi-Hopf algebras in \cite{MasNg:CIFSISQHA} can be applied. 

In principle the same approach, using now the higher indicator formula for quasi-Hopf algebras from \cite{NgSch:CIHISQHA}, could be used to obtain higher indicator formulas for group-theoretical categories. However, those formulas involve iterated applications of the associator elements of the relevant quasi-Hopf algebra dealing with the parentheses of iterated tensor products in the category. Applying them with the explicit quasi-Hopf structure deriving from the data of a group-theoretical fusion category seems a formidable task.

We will take an entirely different approach. The formula from \cite[Thm.~4.1.]{NgSch:FSIESC}, generalizing the ``third formula'' from \cite{KasSomZhu:HFSI}, links higher Frobenius-Schur indicators in a spherical fusion category $\mathcal C$ to the ribbon structure of the Drinfeld center $\CTR(\mathcal C)$ and the functor from $\mathcal C$ to $\CTR(\mathcal C)$ adjoint to the underlying functor. The ``third formula'' was used in \cite{MR2774703} to calculate indicators in Tambara-Yamagami categories; in our context the approach is aided by the fact that the centers of group-theoretical fusion categories are easy to determine: A group-theoretical fusion category is the monoidal category of bimodules over the (twisted) group algebra of a subgroup $H$ of a finite group $G$ inside the category $\Vect_G$ of $G$-graded vector spaces (twisted by a three-cocycle on $G$). By \cite{MR1822847}, the Drinfeld center of such a bimodule category is equivalent to the Drinfeld center of the ``ambient'' category. In different language this means that group-theoretical fusion categories are Morita equivalent to the category of graded vector spaces with twisted associativity; see the survey \cite{MR3077244}. We will treat the case of a group-theoretical fusion category defined without cocycles. Thus $\C=\HM G..H.H.$, the center is $\CTR(\HM G..H.H.)=\CTR(\Vect_G)$, equivalent to the category of modules over the Drinfeld double of $G$.

In a sense the underlying functor $\CTR(\Vect_G)\to\HM G..H.H.$ is already known explicitly from \cite{MR1822847}, but we need to do more. Simple objects in $\HM G..H.H.$ are parametrized by group-theoretical data, namely (equivalence classes of) pairs consisting of an element of $G$ and an irreducible representation of a certain stabilizer subgroup of $H$. Simple objects of $\CTR(\Vect_G)$ are also classified by group-theoretical data, (equivalence classes of) pairs consisting of an element of $G$ and an irreducible representation of its centralizer. In \cref{sec:center-adjoint} we will describe the underlying functor $\CTR(\Vect_G)\to\HM G..H.H.$ on the level of simple objects by a formula involving only the combinatorics and representation theory of subgroups of $G$. Given this description one can turn things around and describe the adjoint functor $\HM G..H.H.\to\CTR(\Vect_G)$ equally explicitly. Admittedly the resulting description, while completely explicit and entirely on the level of groups, subgroups, and group representations, is quite unwieldy --- this is perhaps natural, since one has to deal with how conjugacy classes and centralizers (involved in the description of modules over the Drinfeld double) relate to double cosets of a chosen subgroup, and stabilizers of one-sided cosets under the regular action (involved in the description of $\HM G..H.H.$).

In \cref{sec:indic-form-group} we will use the description of the adjoint functor and the ``third formula'' to obtain a formula for the higher indicators of the simple objects of $\HM G..H.H.$. Luckily we do not need the entire information on the adjoint, but only the traces of the ribbon structure on the images under the adjoint. This allows to dramatically simplify the immediate result based on the complicated description of the adjoint to obtain a surprisingly simple-looking formula for the higher indicators. It is in fact even simpler than Natale's formula for second indicators, and uses only group characters and the combinatorics of group elements and subgroups, without mentioning the associated quasi-Hopf algebra and its characters at all. One should admit, though, that characters of the associated quasi-Hopf algebra are in turn described in more ``basic'' terms in \cite{Nat:FSICFC}. Also, our results are marred by the obvious flaw that they do not treat general group-theoretical categories, but only those in whose definition the relevant group cocycles are trivial --- we will amend this flaw in a future paper.

We also treat variants of the indicator formula that are more complicated, involving passing to orbits under the action of auxiliary subgroups, but computationally advantageous for the same reason that they pass from sums over the entire group $H$ to sums over certain orbits.

In \cref{sec:pedestrian-example} we will explicitly calculate indicators in several examples of fusion categories associated to an inclusion of symmetric groups $S_{n-2}\subset S_n$. We use the ``simple'' version of our indicator formula for the cases $n=4,5$. The cases $n=6,7$ illustrate how the more complicated versions reduce the size of the calculations needed down to a manageable size.
\section{Preliminaries}
\label{sec:preliminaries}
Throughout the paper, $G$ is a finite group, and $H\subset G$ a subgroup. We denote the adjoint action of $G$ on itself by $x\adhit g=xgx\inv$. If $V$ is a representation of a subgroup $K\subset G$, and $x\in G$, we denote by $x\adhit V$ the twisted representation of $x\adhit K$ with the same underlying vector space $V$ on which $y\in x\adhit K$ acts like $x\inv\adhit y\in K$.

We work over the field $\CC$ of complex numbers, representations are complex representations, and characters ordinary characters. 

The category $\HM G..H.H.:=\HM \CC G..\CC H.\CC H.$ is defined as the category of $\CC H$-bimodules over the group algebra of $H$, considered as an algebra in the category of $\CC G$-comodules, that is, of $G$-graded vector spaces. Thus, an object of $\HM G..H.H.$ is a $G$-graded vector space $M\in\Vect_G$ with a two-sided $H$-action compatible with the grading in the sense that $|hmk|=h|m|k$ for $h,k\in H$ and $m\in M$.

The category $\HM G..H.H.$ is a fusion category. The tensor product is the tensor product of $\CC H$-bimodules. Simple objects are parametrized by irreducible representations of the stabilizers of right cosets of $H$ in $G$. More precisely, let $D\in H\lquot G/H$ be a double coset of $H$ in $G$, let $d\in D$, and let $S=\Stab_H(dH)=H\cap (d\adhit H)$ be the stabilizer in $H$ of the right coset $dH$ under the action of $H$ on its right cosets in $G$. Then the subcategory $\HM D..H.H.\subset \HM G..H.H.$ defined to contain those objects the degrees of all of whose homogeneous elements lie in $D$ is equivalent to the category $\Rep(S)$ of representations of $S$. The equivalence $\HM D..H.H.\to\Rep(S)$ takes $M$ to $(M_{dH})/H\cong(M/H)_{dH/H}$, the space of those vectors in the quotient of $M$ by the right action of $H$ whose degree lies in the right coset of $d$. Details are in \cite{Zhu:HARRHA,Sch:HAEMC}. We will denote the inverse equivalence by $\Indtobim_d\colon\Rep(\Stab_H(dH))\to\HM HdH..H.H.$, so that we have a category equivalence 
\begin{equation*}
  \bigoplus_d\Rep(\Stab_H(dH))\xrightarrow{(\Indtobim_d)_d}\HM G..H.H.
\end{equation*}
in which the sum runs over a set of representatives of the double cosets of $H$ in $G$. Of course $\HM D..H.H.$ can be described by choosing a different representative of $D$. If $h\in H$, then $dh$ has the same right coset as $d$, and $\Indtobim_{dh}=\Indtobim_{d}$, while $\Stab_H(hdH)=h\adhit\Stab_H(dH)$ and $\Indtobim_d(W)=\Indtobim_{hd}(h\adhit W)$ for $W\in\Rep(\Stab_H(dh))$.

In the special case $H=G$ the above description, with the neutral element representing the sole class of $G$ in $G$, amounts to the (well-known) equivalence $\Rep(G)\cong\HM G..G.G.$ sending $V\in\Rep(G)$ to $V\ot \CC G$ with the regular right $G$-action and the diagonal left $G$-action. This is a monoidal category equivalence.

The category $\lYD G=\lYD{\CC G}$ of (left-left) Yetter-Drinfeld modules over $\CC G$ has objects the $G$-graded vector spaces with a left $G$-action compatible with the grading in the sense that $|gv|=g|v|g\inv$ for $g\in G$ and $v\in V\in\lYD G$. The category $\lYD G$ is the (right) center of the category $\HM G....$ of $G$-graded vector spaces: The half-braiding $c\colon U\ot V\to V\ot U$ between a graded vector space $U$ and a Yetter-Drinfeld module $V$ is given by $u\ot v\mapsto |u|v\ot u$. To calculate indicators using the ``third formula'' we also need the fact that the canonical pivotal structure of $\lYD G$ is given by the ordinary vector space isomorphism $V\to V^{**}$, so that pivotal trace and ordinary trace coincide. Finally the ribbon automorphism $\theta$ of an object $V\in\lYD G$ is given by $\theta(v)=|v|v$.

 Simple objects of $\lYD G$ are parametrized by irreducible representations of the centralizers in $G$ of elements of $G$. (In fact this can be viewed as a special case of the description of graded bimodules above, as we shall review in \cref{caseofdouble} below). More precisely, let $g\in G$ and $C_G(g)$ the centralizer of $g$ in $G$. Then a functor 
 \begin{equation*}
   \IndtoYD_g\colon\Rep(C_G(g))\to \lYD G
 \end{equation*}
can be defined by sending $V\in\Rep(C_G(g))$ to the $\CC G$-module $\Ind_{C_G(g)}^GV=\CC G\ou{\CC C_G(g)}V$ endowed with the grading given by $|x\ot v|=xgx\inv$ for $x\in G$ and $v\in V$. We note the special case $g=1$ which recovers the canonical (monoidal) inclusion functor $\Rep(G)\to\lYD G$. Summing over different elements we obtain a category equivalence
\begin{equation*}
  \bigoplus_g\Rep(C_G(g))\xrightarrow{(\IndtoYD_g)_g}\lYD G.
\end{equation*}
The sum runs over a set of representatives of the conjugacy classes of $G$, and the image of the functor $\IndtoYD_g$ consists of those Yetter-Drinfeld modules the degrees of whose homogeneous elements lie in the conjugacy class of $g$. We note for later use that the ribbon automorphism of $\IndtoYD_g(V)$ is $\theta(x\ot v)=(x\adhit g)(x\ot v)=xg\ot v=x\ot gv$; the trace of $\theta^m$ is therefore $[G\colon C_G(g)]\chi(g^m)$ if $\chi$ denotes the character of $V$.

As a final piece of notation, we will write $\langle M,N\rangle:=\dim_\CC (\Hom_{\C}(M,N))$ for objects $M,N$ in a semisimple category.
\section{The center and the adjoint}
\label{sec:center-adjoint}

By a result of Müger \cite{MR1966525} the Drinfeld center $\CTR(\C)$ of a pivotal fusion category $\C$ is a modular category, and the underlying functor $\CTR(\C)\to\C$ has a two-sided adjoint $\mathcal K$. To handle the center of $\HM G..H.H.$ and the adjoint functor $\mathcal K$ we use the fact \cite{MR1822847} that the center of a category of bimodules in a tensor category $\C$ coincides, in many cases including the present one, with the center of $\C$ itself.

To be precise, we will use the ``right center'' $\RCTR(\C)$ whose objects are pairs $(V,c)$ in which $c\colon X\ot V\to V\ot X$ is a half-braiding defined for any $X\in\C$, and we denote by $\RADJ$ the adjoint functor of the underlying functor $\RCTR(\C)\to\C$.

Then, writing $\C=\HM G....=\Vect_G$ for the category of $G$-graded vector spaces, we have a category equivalence 
\begin{equation*}
  \lYD G\cong \RCTR(\C)\to\RCTR({_{\CC H}\C_{\CC H}})=\RCTR(\HM G..H.H.)
\end{equation*}
which sends $(N,c)\in\RCTR(\C)$ to an object of $\RCTR({_{\CC H}\C_{\CC H}})$ whose underlying right $\CC H$-module is $N\ot \CC H$, whose left $\CC H$-module structure is given by
\begin{equation*}
  \CC H\ot N\ot \CC H\xrightarrow{c\ot \CC H}N\ot \CC H\ot \CC H\xrightarrow{N\ot\nabla}N\ot \CC H
\end{equation*}
and whose half-braiding (which we do not need) is induced by the half-braiding of $N$.

Thus, we identify $\RCTR(\HM G..H.H.)=\lYD G$, and we identify the underlying functor $\RCTR(\HM G..H.H.)\to\HM G..H.H.$ with the functor 
\begin{equation*}
  \underlie\colon \lYD G\ni N\to N\ot \CC H\in\HM G..H.H.,
\end{equation*}
where the obvious right $\CC H$-module $N\ot \CC H$ has the left module structure $a(n\ot b)=an\ot ab$ and the grading $|n\ot b|=|n|b$.

Next, let $g\in G$, set $C:=C_G(g)$, and let $V\in\Rep(C)$. We consider 
\begin{equation*}
  \underlie\IndtoYD_g(V)=\CC G\ou{\CC C}V\ot \CC H\in\HM G..H.H..
\end{equation*}
Let $\mathfrak X_g$ be a set of representatives of the double cosets in $H\lquot G/C$, so that $G=\disjun_{x\in\mathfrak X_g}HxC$. Then each $\CC HxC\ou{\CC C}V\ot \CC H\subset \CC G\ou{\CC C}V\ot \CC H$ is a subobject in $\HM G..H.H.$, and we have
\begin{equation*}
  \underlie\IndtoYD_g(V)=\bigoplus_{x\in\mathfrak X_g}\CC HxC\ou{\CC C}V\ot \CC H.
\end{equation*}
Note that the degrees of the homogeneous elements of $\CC HxC\ou{\CC C}V\ot H$ lie in the double coset $H(x\adhit g)H$, so that $\CC HxC\ou{\CC C}V\ot \CC H$ is in the image of the functor $\Indtobim_{x\adhit g}$. To calculate the preimage, observe first that the degree of $hxc\ot v\ot h'\in \CC HxC\ou{\CC C}V\ot \CC H$ is $(hx\adhit g)h'$, and thus in $(x\adhit g)H$ iff $h\in\Stab_H((x\adhit g)H)=:J$. Hence
\begin{equation*}
  \CC HxC\ou{\CC C}V\ot \CC H=\Indtobim_{x\adhit g}(\CC JxC\ou{\CC C}V).
\end{equation*}
Next, observe that for $j,\tilde j\in J$ and $c,\tilde c\in C$ we have $jxc=\tilde jx\tilde c$ iff $\tilde j\inv j=x\adhit(\tilde cc\inv)$, which implies that we have an isomorphism
\begin{equation*}
  \CC JxC\ou{\CC C}V\ni jxc\ot v\mapsto j\ot cv\in \CC J\ou{\CC [J\cap(x\adhit C)]}(x\adhit V).
\end{equation*}
Note that $J\cap (x\adhit C)=\Stab_H((x\adhit g)H)\cap C_G(x\adhit g)=H\cap C_G(x\adhit g)=H\cap x\adhit C$.

We have shown:
\begin{equation*}
  \CC HxC\ou{\CC C}V\ot \CC H=\Indtobim_{x\adhit g}\left(\Ind_{H\cap(x\adhit C)}^{\Stab_H((x\adhit g)H)}\Res_{H\cap (x\adhit C))}^{x\adhit C}(x\adhit V)\right), 
\end{equation*}
whence
\begin{equation*}
  \underlie\IndtoYD_g(V)=\bigoplus_{x\in\mathfrak X_g}\Indtobim_{x\adhit g}\left(\Ind_{H\cap(x\adhit C)}^{\Stab_H((x\adhit g)H)}\Res_{H\cap (x\adhit C)}^{x\adhit C}(x\adhit V)\right).
\end{equation*}

Let $d\in G$ and $S=\Stab_H(dH)$. Let $\mathfrak H_d$ be a set of representatives of $H/S$. Thus the double coset $HdH$ is the disjoint union $HdH=\bigsqcup_{h\in \mathfrak H_d}hdH$, that is $\mathfrak H_dd$ is a set of representatives of the right cosets contained in $HdH$.

If $x\adhit g\in HdH$, then there is a unique $h\in\mathfrak H_d$ such that $(x\adhit g)H=hdH$, thus $\Stab_H((x\adhit g)H)=\Stab_H(hdH)=h\adhit S$, and for a representation $N$ of $\Stab_H((x\adhit g)H)$ we have $\Indtobim_{x\adhit g}N=\Indtobim_{hd}N=\Indtobim_d(h\inv\adhit N)$. Again $H\cap (x\adhit C)=(h\adhit S)\cap (x\adhit C)$. Thus
\begin{align*}\left(\underlie\IndtoYD_g(V)\right)_{HdH} &=\bigoplus_{\substack{x\in\mathfrak X_g\\h\in\mathfrak H_d\\x\adhit g\in hdH}}\Indtobim_{hd}\left(\Ind_{(h\adhit S)\cap(x\adhit C)}^{h\adhit S}\Res^{x\adhit C}_{(h\adhit S)\cap(x\adhit C)}(x\adhit V)\right)\\
&=\bigoplus_{\substack{x\in\mathfrak X_g\\h\in\mathfrak H_d\\x\adhit g\in hdH}}\Indtobim_d\left(h\inv\adhit\left(\Ind_{(h\adhit S)\cap(x\adhit C)}^{h\adhit S}\Res^{x\adhit C}_{(h\adhit S)\cap(x\adhit C)}(x\adhit V)\right)\right),
\end{align*}
and if $W\in\Irr(S)$, then 
\begin{multline*}
  \left\langle\underlie\IndtoYD_g(V),\Indtobim_d(W)\right\rangle
\\=\sum_{\substack{x\in\mathfrak X_g\\h\in\mathfrak H_d\\x\adhit g\in hdH}}\left\langle h\inv\adhit\left(\Ind_{(h\adhit S)\cap(x\adhit C)}^{h\adhit S}\Res^{x\adhit C}_{(h\adhit S)\cap(x\adhit C)}(x\adhit V)\right),W\right\rangle
\\=\sum_{\substack{x\in\mathfrak X_g\\h\in\mathfrak H_d\\x\adhit g\in hdH}}\left\langle \Ind_{(h\adhit S)\cap(x\adhit C)}^{h\adhit S}\Res^{x\adhit C}_{(h\adhit S)\cap(x\adhit C)}(x\adhit V),h\adhit W\right\rangle.
\end{multline*}
For the adjoint $\RADJ$ of $\underlie$ this implies, by Frobenius reciprocity:
\begin{multline*}
  \left\langle \RADJ\Indtobim_d(W),\IndtoYD_g(V)\right\rangle
\\=\sum_{\substack{x\in\mathfrak X_g\\h\in\mathfrak H_d\\x\adhit g\in hdH}}\left\langle x\inv\adhit\left(\Ind_{(h\adhit S)\cap(x\adhit C)}^{x\adhit C}\Res^{h\adhit S}_{(h\adhit S)\cap(x\adhit C)}(h\adhit W)\right),V\right\rangle.
\end{multline*}
This means that we have calculated a formula for the adjoint $\RADJ$: denoting by $\mathfrak C$ a system of representatives for the conjugacy classes of $G$, we have
\begin{align*}
  \RADJ&\Indtobim_d(W)\\&=\sum_{\substack{g\in\mathfrak C\\x\in\mathfrak X_g\\h\in\mathfrak H_d\\x\adhit g\in hdH}}\IndtoYD_g\left(x\inv\adhit\left(\Ind_{(h\adhit S)\cap(x\adhit C_G(g))}^{x\adhit C_G(g)}\Res^{h\adhit S}_{(h\adhit S)\cap(x\adhit C_G(g))}(h\adhit W)\right)\right)\\
&=\sum_{\substack{g\in\mathfrak C\\x\in\mathfrak X_g\\h\in\mathfrak H_d\\x\adhit g\in hdH}}\IndtoYD_{x\adhit g}\left(\Ind_{(h\adhit S)\cap(x\adhit C_G(g))}^{x\adhit C_G(g)}\Res^{h\adhit S}_{(h\adhit S)\cap(x\adhit C_G(g))}(h\adhit W)\right).
\end{align*}
While this is clearly not a particularly pleasant or practical formula, we can say something in its favor: It expresses the functor $\RADJ$ entirely in terms of the groups involved and their representations, using, of course, the translation of group representations to objects in the two categories involved via the functors $\Indtobim$ and $\IndtoYD$.
\section{Indicator formulas for group inclusions}
\label{sec:indic-form-group}

We retain the notations of the previous section, and proceed to calculate the higher Frobenius-Schur indicators of objects in $\HM G..H.H.$. This is based on the categorical version of the ``third formula'' in \cite{KasSomZhu:HFSI} that calculates indicators in a fusion category $\C$ through the adjoint $\RADJ$

The formula obtained above for the adjoint $\RADJ\colon \HM G..H.H.\to \lYD G$ yields, via \cite[Thm. 4.1]{NgSch:FSIESC}, a formula for the higher indicators of the simple objects of $\HM G..H.H.$. Since we are dealing with the right center, the relevant formula, \cite[Rem.~4.3]{NgSch:FSIESC}, is
\begin{equation*}
  \nu_m(X)=\frac{1}{|G|}\Tr(\theta_{\RADJ(X)}^{-m}).
\end{equation*}
We proceed to use the information available on $\RADJ$ to apply it.

First, let $\eta'$ be a character of $(h\adhit S)\cap(x\adhit C)$, and $\chi=\Ind_{(h\adhit S)\cap(x\adhit C)}^{x\adhit C}(\eta')$. Then by a standard formula for induced characters
\begin{align*}
  \chi(x\adhit g^m)&=\frac1{|(h\adhit S)\cap(x\adhit C)|}\sum_{\substack{y\in x\adhit C\\y\adhit x\adhit g^m\in h\adhit S}}\eta'(y\adhit x\adhit g^m)\\
&=
\begin{cases}
  [x\adhit C\colon (h\adhit S)\cap (x\adhit C)]\eta'(x\adhit g^m)&\text{if }x\adhit g^m\in h\adhit S\\
0&\text{otherwise,}
\end{cases}
\end{align*}
as elements in $x\adhit C$ commute with $x\adhit g^m$.

Let $\eta$ be the character of $W\in\Rep(S)$, and let $\chi$ be the character of $V:=\Ind_{h\adhit S\cap x\adhit C}^{x\adhit C}\Res^{h\adhit S}_{h\adhit S\cap x\adhit C}(h\adhit \eta)$. Then
\begin{align*}
  \Tr(\theta_{\IndtoYD_{x\adhit g}(V)}^m)
     &=[G\colon x\adhit C]\chi(x\adhit g^m)\\
     &=
     \begin{cases}
       [G\colon (h\adhit S)\cap(x\adhit C)]\eta(h\inv x\adhit g^m)&\text{if }x\adhit g^m\in h\adhit S\\
       0&\text{otherwise.}
     \end{cases}
\end{align*}
By the formula for $\RADJ(\Indtobim_d(W))$ obtained in the previous section, this finally implies (using $|(h\adhit S)\cap(x\adhit C_G(g))|=|S\cap (h\inv x\adhit C_G(g)|=|S\cap C_G(h\inv x\adhit g)|$)
\begin{align}
\label{eq:15}
  \nu_m(\Indtobim_d(W))
    &=\sum_{\substack{g\in\mathfrak C\\x\in\mathfrak X_g\\h\in\mathfrak H_d\\x\adhit g\in hdH\\x\adhit g^m\in h\adhit S}}\frac1{|S\cap C_G(h\inv x\adhit g)|}\ol\eta(h\inv x\adhit g^m).
\end{align}
Surely this sum is not pleasant to work with; it involves summing over all conjugacy classes of the group and all representatives of certain double cosets, as well as over the coset representatives in $\mathfrak H_d$, albeit that last sum involves either no summand (for many combinations of $g$ and $x$ we might have $x\adhit g\not\in HdH$), or just one summand (the representative of the unique right coset containing $x\adhit g$).

We shall process it further using the observation
\begin{equation}
  \label{eq:3}
  HdH=\bigsqcup_{\substack{g\in\mathfrak C\\x\in\mathfrak X_g\\x\adhit g\in HdH}}H\adhit(x\adhit g)
  =\bigsqcup_{\substack{g\in\mathfrak C\\x\in\mathfrak X_g\\h\in\mathfrak H_d\\x\adhit g\in hdH}}H\adhit (h\inv x\adhit g)
\end{equation}
For the first equality, one has to check when $x\adhit g$ and $y\adhit g$, for $x,y\in G$, are in the same orbit of the action of $H$ on $G$ by conjugation:
\begin{multline*}
  \exists h\in H\colon h\adhit(x\adhit g)=y\adhit g
\Leftrightarrow \exists h\in H\colon hxgx\inv h\inv=ygy\inv\\
\Leftrightarrow\exists h\in H\colon y\inv hx\in C_G(g)
\Leftrightarrow x\in HyC_G(g),
\end{multline*}
while the second is an obvious reparametrization.

Thus, the set
\begin{equation}
\label{eq:Rd}
  \mathfrak R_d=\{h\inv x\adhit g|g\in\mathfrak C,x\in\mathfrak X_g,h\in\mathfrak H_d,x\adhit g\in hdH\}
\end{equation}
is a set of representatives of the orbits of the action of $H$ on $HdH$ by conjugation. Moreover, $\mathfrak R_d\subset dH$. Thus, $\mathfrak R_d$ is a set of representatives of the orbits of the action of $S$ on $dH$ by conjugation. We have very nearly proved the main result of the paper:
\begin{Thm}\label{indicatorformula}
  Let $G$ be a finite group, $H\subset G$ a subgroup, $d\in G$, $S=\Stab_H(dH)$, $W\in\Rep(S)$ with character $\eta$, and $\Indtobim_d(W)$ the object of $\HM G..H.H.$ corresponding to $W$.
Then
\begin{equation}
  \label{eq:14}
  \nu_m(\Indtobim_d(W))=\frac1{|S|}\sum_{\substack{r\in dH\\r^m\in S}}\ol\eta(r^m)=\frac1{|S|}\sum_{\substack{h\in H\\(dh)^m\in S}}\ol\eta((dh)^m).
\end{equation}
\end{Thm}
\begin{proof}
Substituting \cref{eq:Rd} in the indicator formula \cref{eq:15} yields 
  \begin{equation}
    \label{eq:4}
    \nu_m(\Indtobim_d(W))=\sum_{\substack{r\in\mathfrak R_d\\r^m\in S}}\frac1{|S\cap C_G(r)|}\ol\eta(r^m).
  \end{equation}
But for $s\in S$ we have $(s\adhit r)^m\in S\Leftrightarrow r^m\in S$, and $\eta((s\adhit r)^m)=\eta(r^m)$ whenever $r^m\in S$. Since $S\cap C_G(r)$ is the stabilizer of $r$ under the adjoint action of $S$, the first equality in \eqref{eq:14} follows. The second equality is a trivial reparametrization.
\end{proof}
In the following we keep the notations of \cref{indicatorformula}.
\begin{Rem}
  Note that for $r\in dH$ we have $r^m\in S\Leftrightarrow r^m\in H$. Thus we could modify the conditions in the sums \eqref{eq:14} and subsequent similar sums, but in the examples that we treated it seemed easier to check whether an element is in $S$ than to check whether it is in $H$.
\end{Rem}

\begin{Rem}
  The elements 
  \begin{equation}
    \label{eq:7}
    \mu_m(d):=\frac1{|S|}\sum_{\substack{r\in dH\\r^m\in S}}r^m=\frac1{|S|}\sum_{\substack{h\in H\\(dh)^m\in S}}(dh)^m\in \CC S
  \end{equation}
  for $m\in \N$ are central in the group algebra $\CC S$, and $\nu_m(\Indtobim_d(W))=\eta(\mu_m(d))$.
\end{Rem}
\begin{Rem}
\label{ifdcentral}
  If $d\in C_G(H)$, then $S=H$, and for $h\in H$ we have $(dh)^m=d^mh^m\in H$ if and only if $d^m\in H$, so that
  \begin{equation}
    \label{eq:8}
    \mu_m(d)=
    \begin{cases}
      d^m\frac1{|H|}\sum_{h\in H}h^m&\text{if }d^m\in H\\
      0&\text{otherwise,}
    \end{cases}
  \end{equation}
  and therefore, since $d^m\in H$ is in the center of $H$:
  \begin{equation}
    \label{eq:9}
    \nu_m(\Indtobim_d(W))=
    \begin{cases}
      \frac{\ol\eta(d^m)}{\eta(1)}\nu_m(W)&\text{if }d^m\in H\\
      0&\text{otherwise.}
    \end{cases}
  \end{equation}
  The most obvious case of this is when $d=1$; the image of $\Indtobim_1$ is the monoidal subcategory $\HM H..H.H.\subset\HM G..H.H.$, which is monoidally equivalent to $\Rep(H)$. The formula \eqref{eq:9} can also be used to easily obtain examples where the higher indicators are not real: The cyclic group $G$ of order $9$, its generator $d$, its subgroup $H$ of order $3$ and a nontrivial irreducible character of the latter will do to obtain $\nu_3(\Indtobim_d(W))$ a nontrivial third root of unity.
\end{Rem}

\begin{Lem}
  Let $y\in S$. Then
  \begin{equation}
    \label{eq:2}
    \sum_{\chi\in\Irr(S)}\nu_m(\Indtobim_d(\chi))\chi(y)=|\{h\in H|(dh)^m=y\}|.
  \end{equation}
\end{Lem}
In fact the function $\zeta_m(y)=|\{h\in H|(dh)^m=y\}|$ is easily seen to be a class function on $S$, so one can verify \eqref{eq:2} by taking its scalar product with an irreducible character $\eta$. The left hand side gives the $m$-th indicator by the orthogonality relations, the right hand side by \eqref{eq:14}.

\begin{Rem}
  Assume that $H\subset G$ is part of an exact factorization, i.\ e.\ there exists a subgroup $L\subset G$ such that $LH=G$ and $L\cap G=\{1\}$. As pointed out in \cite{Sch:HBCESK}, the category $\HM G..H.H.$ is then equivalent to the category of modules over a bismash product Hopf algebra $\CC ^L\# \CC H$. Thus, our results comprise a method to calculate indicators for bismash product Hopf algebras (of which the double below is a special case).
\end{Rem}

\begin{Expl}\label{caseofdouble}
Let $\Gamma$ be a finite group, $G=\Gamma\times \Gamma$, and $H=\Delta(\Gamma)$, where $\Delta\colon \Gamma\to\Gamma\times\Gamma$ is the diagonal embedding. It is well known that the category $\HM G..H.H.\cong\HM\Gamma.\Gamma.\Gamma.\Gamma.$ is equivalent to the module category of the Drinfeld double of $\Gamma$ (in fact this is a special case of \cite{MR95j:16047}).

Let $\mathfrak G$ be a cross section of the conjugacy classes of $\Gamma$. Then $\{(\gamma,1)|\gamma\in\mathfrak G\}$ is a cross section of the double cosets of $H$ in $G$. Let $d=(\gamma,1)$. Then
$S=\Stab_H(dH)=\Delta(C_\Gamma(\gamma))$. Let $h=\Delta(\theta)\in H$ and $m\in\N$. Then $(dh)^m=(\gamma\theta,\theta)^m=((\gamma\theta)^m,\theta^m)$, and thus $(dh)^m\in S$ if and only if $(\gamma\theta)^m=\theta^m$. Thus our indicator formula yields
\begin{equation}
  \label{eq:5}
\nu_m(\Indtobim_d(W))=\frac1{|C_\Gamma(\gamma)|}\sum_{\substack{\theta\in\Gamma\\(\gamma\theta)^m=\theta^m}}\ol\eta(\theta^m).
\end{equation}
This formula was obtained in \cite{KasSomZhu:HFSI}, see also \cite{2012arXiv1208.4153I}, where the corresponding special case of \eqref{eq:2} can be found. Note that we can replace $\ol\eta$ by $\eta$ since the indicators in this case are known to be real.
\end{Expl}

In the proof of \cref{indicatorformula} we have obtained the simple looking indicator formula \eqref{eq:14} via the more complicated formula \eqref{eq:4}. But in fact the latter is, in some respects, better than the former: It involves a sum over less terms, namely orbits of the adjoint action of $S$ instead of individual elements of $dH$. Of course for this simplification we could have taken any section of the orbits on $dH$ instead of $\mathfrak R_d$. But in fact we can also pass to orbits over a group different from $S$; also, it may be convenient to take orbits in $H$ of the action on $H$ corresponding to the adjoint action on $dH$:
\begin{Prop}
  \label{prop:betterorbits}
  Continuing in the notations of \cref{indicatorformula}, set $E=C_G(d)\cap SC_G(S)\cap N_G(H)$. Then $SE=ES$ is a subgroup of $G$. Let $S'\subset SE$ be a subgroup, and let $\mathfrak R'_d$ be a section of the orbits of $dH$ under the adjoint action of $S'$ on $dH$. Then 
  \begin{equation}
    \label{eq:16}
    \nu_m(\Indtobim_d(W))%=\sum_{\substack{r\in\mathfrak R'_d\\r^m\in H}}\frac1{|S'\cap C_G(r)|}\eta(r^m)    
=\frac 1{|S|}\sum_{\substack{r\in\mathfrak R'_d\\r^m\in S}}\frac{|S'|}{|S'\cap C_G(r)|}\ol\eta(r^m).
  \end{equation}
Alternatively, let $S'$ act on $H$ by the ``twisted conjugation'' defined by $s\twisthit h=(d\inv\adhit s)hs\inv$. Let $\mathfrak T'_d$ be a system of representatives of the orbits. Then 
  \begin{equation}
    \label{eq:11}
    \nu_m(\Indtobim_d(W))%=\sum_{\substack{r\in\mathfrak R'_d\\r^m\in H}}\frac1{|S'\cap C_G(r)|}\eta(r^m)    
=\frac 1{|S|}\sum_{\substack{h\in\mathfrak T'_d\\(dh)^m\in S}}\frac{|S'|}{|S'\cap C_G(dh)|}\ol\eta((dh)^m).
  \end{equation}
\end{Prop}
\begin{proof}
Let $x\in E$ and $u\in S=H\cap (d\adhit H)$. Then $x\adhit u\in(x\adhit H)\cap (xd\adhit H)=H\cap d\adhit H=S$ since $x\adhit H=H$ and $xd=dx$ by hypothesis. Thus $E$ normalizes $S$, and $SE=ES$ is a subgroup of $G$. Now let $x\in E$ and $h\in H$. Since $x\in SC_G(S)$, we have $(dh)^m\in S$ if and only if $x\adhit(dh)^m\in S$; in fact these two elements are then conjugate in $S$. The condition $x\in C_G(d)$ implies $x\adhit(dh)^m=(d(x\adhit h))^m$, and $x\in N_G(H)$ implies $x\adhit h\in H$. Thus the action of $S'$ on $dH$ is well defined, and the condition $r^m\in S$ is invariant along the orbits as well as the values $\eta(r^m)$ along those orbits where $r^m\in S$. This implies \eqref{eq:16}, since $S'\cap C_G(r)$ is the stabilizer of $r$. Since $s\adhit (dh)=d(s\twisthit h)$ for $s\in S$ and $h\in H$, we obtain \eqref{eq:11} by a simple reparametrization.
\end{proof}
\begin{Rem}
The previous result is perhaps the most useful if $S'\subset C_G(d)$, so that the twisted adjoint action coincides with the adjoint action. At any rate it allows to replace $H$ by a set of orbit representatives before passing to the nastier part of the calculations involved in applying the indicator formula to concrete examples.

To set notations for subsequent calculations, let $\ol G$ be the set of orbits of $G$ under the adjoint action of $S'$, and $\ol S$ the image of $S$ in $\ol G$. We do not distinguish notationally elements of $\ol G$ from those of $G$. We also let $\tilde H$ be the set of orbits of the twisted adjoint action of $S'$ on $H$, and $Q(d):=\sum_{h\in H}h\in \CC \tilde H$. Set
\begin{equation}
  \label{eq:12}
  T(d):=\sum_{h\in\mathfrak T'_d}[S'\colon S'\cap C_G(dh)]dh=dQ(d)\in \CC \ol G.
\end{equation}
Let $\CC \ol G\ni x\mapsto x^{[m]}\in \CC \ol G$ be the linear map induced by taking $m$-th powers of group elements. Let $\pi\colon \CC \ol G\to \CC \ol S$ be the linear projection annihilating $\ol G\setminus \ol S$. Then
\begin{align}
  \label{eq:13}
  \nu_m(\Indtobim_d(W))&=\ol\eta(\ol\mu_m(d))&\text{with }&&\ol\mu_m(d)=\frac1{|S|}\pi(T(d)^{[m]}).
\end{align}
Of course $\ol\mu_m(d)$ is just the image of $\mu_m(d)$ in $\CC \ol S$.
\end{Rem}
% Next, let $g\in G$. Let $C:=C_G(g)$ be the centralizer of $g$ in $G$. Let $x_i$ be a set of double coset representatives of $H,C$ in $G$, that is, $G=\disjun_{i=1}^{\beta}Hx_iC$. 
% Some of the conjugates $x_i\adhit g$ will be in the double coset $HdH$. Without loss of generality, say $x_i\adhit g\in r_{j(i)}H$ for $i=1,\dots,\gamma$, with $j(i)\in\{1,\dots,\alpha\}$, while $x_i\adhit g\not\in HdH$ for $i>\gamma$. Put $y_i=a_{j(i)}\inv x_i$ for $i=1,\dots,\gamma$.

% \begin{Prop}
%   With the above notations, let $V$ be an irreducible representation of $C$, and $\hat V$ the corresponding object in $\lYD{G}$. Let $W$ be an irreducible representation of $S$, and $M(W)$ the corresponding object of $\HM G..H.H.$. Then
%   \begin{equation}
% %    \label{eq:1}
%     \left\langle U(\hat V),M(W)\right\rangle=\sum_{i=1}^{\gamma}\left\langle\Res^S_{S\cap y_i\adhit C}W,\Res^{y_i\adhit C}_{S\cap y_i\adhit C}(y_i\adhit V)\right\rangle
%   \end{equation}
% \end{Prop}
% \begin{proof}
  
% \end{proof}
% \begin{Rem}
%   Let $1\leq i\leq\gamma$. For some $b_i\in H$, we have $x_i\adhit g=a_{j(i)}db_i$, and thus $y_i\adhit g=a_{j(i)}\inv (x_i\adhit g)a_{j(i)}=db_ia_{j(i)}\in dH$.
% \end{Rem}

\section{Example calculations}
\label{sec:pedestrian-example}

Consider the symmetric group $S_n$ and the subgroup $S_m\subset S_n$ for $m<n$. For $d\in S_n$ the stabilizer $\Stab_{S_m}(dS_n)=S_m\cap d\adhit S_m$ consists of those permutations $\sigma\in S_m$ for which $d\inv\adhit \sigma\in S_m$. For $d\inv\adhit \sigma$ to fix every element greater than $m$ it is necessary and sufficient that $\sigma$ fix every element $k $ with $d\inv(k)\not\in\{1,\dots,m\}$. Thus $\Stab_{S_m}(dS_m)=S_{\{1,\dots,m\}\cap \{d(1),\dots,d(m)\}}$ is a symmetric group. We have seen that in general higher indicators for the objects of $\HM G..H.H.$ are nonnegative rational linear combinations of character values of the stabilizers $\Stab_H{dH}$. Moreover, higher indicators for any pivotal fusion category are cyclotomic integers. Thus
\begin{Prop}
  Let $m<n$. Then all values of the higher Frobenius-Schur indicators for the objects of $\HM S_n..S_m.S_m.$ are integers.
\end{Prop}

The following example shows that this can fail if we embed $S_m$ into $S_n$ in a different fashion.

\begin{Expl}
  Consider
  \begin{equation*}
    G=S_9\supset H=\{\sigma\in S_9|i\equiv j(3)\Rightarrow\sigma(i)\equiv\sigma(j)(3)\},
  \end{equation*}
  so $H$ is the subgroup of those permutations in $S_9$ that preserve conjugacy modulo $3$. Thus $H\cong S_3$ is generated by $t=(123)(456)(789)$ and $s=(12)(45)(78)$.

  The element $d=(147258369)\in S_9$ satisfies $d^3=t$, so in particular $d\inv\adhit t\in H$. On the other hand $d\inv\adhit s=(12)(45)(79)$, so $d\inv\adhit s\not\in H$ because $1\equiv 7(3)$ while $2\not\equiv 9(3)$. It follows that $S=\Stab_H(dH)=\langle t\rangle.$

  To compute $\mu_3(d)$, observe $d^3=(dt)^3=(dt^2)^3=t$. The computation $ds=(157369)(248)$ and $(ds)^3=(13)(56)(79)\not\in S$ shows that $(dh)^3\not\in S$ for $h\in H\setminus\{1,t,t^2\}$, since such $h$ are conjugate to $s$ by powers of $t$, which commute with $d$. Thus $\mu_3(d)=t$. 

  In particular $\nu_3(\Indtobim_d(\eta))=\zeta\inv$ is not real when $\eta(t)=\zeta$ is a nontrivial third root of unity.
\end{Expl}

We will now compute some of the indicator values for the canonically embedded subgroups $S_{n-2}\subset S_n$ (as we shall see, this contains in a sense the case $S_{n-1}\subset S_n$, or rather $S_{n-2}\subset S_{n-1}$). We note already that all the indicator values we will find are nonnegative.

For $n\geq 4$, it is easy to check that $S_{n-2}$ has the following seven double cosets in $S_n$:
\begin{align*}
  \{\sigma\in S_n&|\sigma(n-1)=n-1,\sigma(n)=n)\}=S_{n-1}\\
  \{\sigma\in S_n&|\sigma(n-1)\neq n-1,\sigma(n)=n)\}\\
  \{\sigma\in S_n&|\sigma(n-1)=n-1,\sigma(n)\neq n)\}\\
  \{\sigma\in S_n&|\sigma(n-1)=n,\sigma(n)=n-1)\}\\
  \{\sigma\in S_n&|\sigma(n-1)=n,\sigma(n)\neq n-1)\}\\
  \{\sigma\in S_n&|\sigma(n-1)\neq n,\sigma(n)=n-1)\}\\
  \{\sigma\in S_n&|\{\sigma(n-1),\sigma(n)\}\cap\{n-1,n\}=\emptyset\}.
\end{align*}
A convenient set of double coset representatives is $d_1=()$, $d_2=(n-2,n-1)$, $d_3=(n-2,n)$, $d_4=(n-1,n)$,  $d_5=(n-2,n-1,n)$, $d_6=(n-2,n,n-1)$, $d_7=(n-3,n-1)(n-2,n)$.

Note that $d_2$ and $d_3$ are conjugate by $(n-1,n)$. The same holds for $d_5$ and $d_6$. We have $\Stab_{S_{n-2}}(d_2S_{n-2})=\Stab_{S_{n-2}}(d_5S_{n-2})=S_{n-3}$, $\Stab_{S_{n-2}}(d_7S_{n-2})=S_{n-4}$, and $\Stab_{S_{n-2}}(d_4S_{n-2})=S_{n-2}$. 

Note that every $d_i$ commutes with the elements in $\Stab_{S_{n-2}}(d_iS_{n-2})$; this is particular to our choice of representatives. It implies that the twisted conjugation action of the stabilizers on the group $S_{n-2}$ from \cref{prop:betterorbits} is the ordinary adjoint action.

Note further that $d_4$ commutes with the elements of $S_{n-2}$. By \cref{ifdcentral} it follows that 
\begin{equation}
  \label{eq:10}
  \nu_m(\Indtobim_{(n-1,n)}(W))=
  \begin{cases}
    \nu_m(W)&\text{if $m$ is even,}\\
    0&\text{if $m$ is odd,}
  \end{cases}
\end{equation}
for any $W\in\Rep(S_{n-2})$, while $\nu_m(\Indtobim_{()}(W))=\nu_m(W)$.

Note also that $d_2\in S_{n-1}$. Thus, the indicators for objects in $\Indtobim_{d_2}(\Rep(S_{n-2}))$ can also be viewed as indicators in the subcategory $\HM S_{n-1}..S_{n-2}.S_{n-2}.$. The subgroup $S_{n-2}\subset S_{n-1}$ is part of an exact factorization, $S_{n-1}=C_{n-1}\cdot S_{n-2}$, where $C_{n-1}$ denotes the cyclic group generated by the $(n-1)$-cycle $(1,2,\dots,n-1)$. As remarked already, these indicators are indicators for modules over a bismash product Hopf algebra $\CC ^{C_{n-2}}\# \CC S_{n-1}$. Observe that the exact factorization suggests a different choice of coset representative, namely the $(n-1)$-cycle instead of $d_2$. We have the feeling that $d_2$ is the better choice since the $(n-1)$-cycle does not commute with elements in the corresponding stabilizer.

Since the images of $\Indtobim_{d_2}$ and $\Indtobim_{d_3}$ are mapped to each other by an autoequivalence, as well as the images of $\Indtobim_{d_5}$ and $\Indtobim_{d_6}$, we can concentrate on the indicators of the objects in the images of $\Indtobim_{d_i}$ for $i=2,5,7$. We will treat some of them below for small values of $n$.

\subsection{$S_2\subset S_4$}
\label{sec:s_2subset-s_4}

Consider $H=\langle(1\;2)\rangle\subset G=S_4$. We have the following double coset representatives, with their right cosets and double cosets:
\begin{equation*}
  \label{eq:6}
  \begin{array}[c]{|r|c|l|l|c|}
\hline
    i&d_i&d_iH\setminus\{d_i\}&Hd_iH\setminus d_iH&\Stab_H(d_iH)\\
    \hline
    1&()&(1\;2)&&H\\
    \hline    
    2&(2\;3)&(1\;2\;3)&(1\;3),(1\;3\;2)&\{()\}\\
\hline
 3&(2\;4)&(1\;2\;4)&(1\;4),(1\;4\;2)&\{()\}\\
\hline
    4&(3\;4)&(1\;2)(3\;4)&&H\\
\hline
    5&(2\;3\;4)&(1\;2\;3\;4)&(1\;3\;4),(1\;3\;4\;2)&\{()\}\\
\hline
6&(2\;4\;3)&(1\;2\;4\;3)&(1\;4\;3),(1\;4\;3\;2)&\{()\}\\
\hline
7&(2\;3)(1\;4)&(1\;4\;2\;3)&(1\;4\;2\;3),(1\;3\;
2\;4)&\{()\}\\
\hline
  \end{array}
\end{equation*}
We proceed to list the sequences of the higher Frobenius-Schur indicators for all the simple objects of $\HM G..H.H.$ in the images of the functors $\Indtobim_{d_i}$. These sequences are periodic and we list them for one complete period:

For $d_1$, they are the sequences of the higher Frobenius-Schur indicators of the representations of $H$, namely $(1,\dots)$ with period one for the trivial, and $(1,0,\dots)$ with period two for the nontrivial representation.

In all other cases, the only powers of the elements of $d_iH$ that lie in the stabilizer $\Stab_H(d_iH)$ are identity elements. (This requires only a glance for $d_4$, as the stabilizer itself is trivial in the other cases.) Thus (regardless of the choice of representation also in the $d_4$ case), the indicator $\nu_m$ counts how many of the two $m$-th powers of the two elements of $d_iH$ are trivial; the count is then divided by two in the $d_4$ case. Thus the indicator sequences, up to a full period, are:
\begin{align*}
  (\nu_m(\Indtobim_{d_i}(W)))_m=
  \begin{cases}
    (0,1,1,1,0,2,\dots)&\text{for }i=2,3\\
    (0,1,\dots)&\text{for }i=4\\
    (0,0,1,1,0,1,0,1,1,0,0,2,\dots)&\text{for }i=5,6\\
    (0,1,0,2,\dots)&\text{for }i=7.
  \end{cases}
\end{align*}
(Note that the case $d_4$ was already treated above using \cref{ifdcentral}.)
\subsection{$S_3\subset S_5$}
\label{sec:s_3subset-s_5}
In this case we have the right cosets
\begin{align*}
\begin{array}[c]{|r|c|l|}\hline
  i&d_i&d_iS_3\setminus\{d_i\}\\%&\Stab_H(d_iH)\\
  \hline
  1&()&(1 2),(1 3),(2 3),(1 2 3),(1 3 2)\\
  \hline  2&(3 4)&(1 2)(3 4),(1 4 3),(2 4 3),(1 2 4 3),(1 4 3 2)\\
\hline
3&\multicolumn{2}{l|}{\text{conjugate preceding row by }(4 5)}\\
\hline
4&(4 5)&(1 2)(4 5),(1 3)(4 5),(2 3)(4 5),(1 2 3)(4 5),(1 3 2)(4 5)\\
\hline
5&(3 4 5)&(1 2)(3 4 5),(1 4 5 3),(2 4 5 3),(1 2 4 5 3),(1 4 5 3 2)\\
\hline
6&\multicolumn{2}{l|}{\text{conjugate preceding line by }(4 5)}\\
\hline
7&(2 4 3 5)&(1 4 3 5 2),(1 5 2 4 3),(2 5)(3 4),(1 4 3)(2 5),(1 5 2)(3 4)\\
\hline
\end{array}
\end{align*}
We have
\begin{equation*}
  \Stab_{S_3}(d_iS_3)=
  \begin{cases}
    S_3&\text{for }i=1,4\\
    S_2=\langle(1\;2)\rangle&\text{for }i=2,3,5,6\\
    \{()\}&\text{for }i=7.
  \end{cases}
\end{equation*}

As indicated above, we will only treat the indicators for $d_2,d_5$ and $d_7$.

One sees that for $i=2$ the only possibility for a power of an element of $d_iS_3$ to be in $\Stab_{S_3}(d_iS_3)$ is if that power is trivial. The same is of course true for $i=7$. So the $m$-th indicators for the simple objects in the images of $\Indtobim_{d_i}$ for $i=2,7$ do not ``see'' the representations of $\Stab_{S_3}(d_iS_3)$, but only count the number of elements whose orders divide $m$; the count has to be divided by $2$ if $i=2$. We have
\begin{align*}
  \nu_m(\Indtobim_{d_2}(W))&=
  \begin{cases}
    0&\text{when }(m,12)=1\\
    1&\text{when }(m,12)=2,3\\
    2&\text{when }(m,12)=4,6\\
    3&\text{when }(m,12)=12.
  \end{cases}\\
  \nu_m(\Indtobim_{d_7}(W))&=
  \begin{cases}
    0&\text{when }(m,60)=1,3\\
    1&\text{when }(m,60)=2\\
    2&\text{when }(m,60)=4,5,15\\
    3&\text{when }(m,60)=6,10\\
    4&\text{when }(m,60)=12,20\\
    5&\text{when }(m,60)=30\\
    6&\text{when }(m,60)=60.
  \end{cases}
\end{align*}

Finally $d_5S_3$ contains one element, $(1\;2)(3\;4\;5)$, whose third power is in $\Stab_{S_3}(d_5S_3)\setminus\{()\}$. Powers of the other elements are only in the stabilizer when they are trivial. Thus we obtain
\begin{equation*}
  \mu_m(d_5)=\mu_m(d_6)=
  \begin{cases}
    0&\text{when }(m,60)=1,2\\
    \frac12(()+(1\;2))&\text{when }(m,60)=3\\
    ()&\text{when }(m,60)=4,5,6,10\\
    2()&\text{when }(m,60)=12,20,30\\
    \frac12(3()+(1\;2))&\text{when }(m,60)=15\\
    3()&\text{when }(m,60)=60
  \end{cases}
\end{equation*}
For the trivial representation $W_0$ of $\langle(1\;2)\rangle$ this yields
\begin{equation*}
  \nu_m(\Indtobim_{d_5}(W_0))=\nu_m(\Indtobim_{d_6}(W_0)=
  \begin{cases}
    0&\text{when }(m,60)=1,2\\
    1&\text{when }(m,60)=3,4,5,6,10\\
    2&\text{when }(m,60)=12,15,20,30\\
    3&\text{when }(m,60)=60.
  \end{cases}
\end{equation*}
For the nontrivial irreducible representation $W_1$ of $\langle(1\;2)\rangle$ we obtain
\begin{equation*}
  \nu_m(\Indtobim_{d_5}(W_1))=\nu_m(\Indtobim_{d_6}(W_1)=
  \begin{cases}
    0&\text{when }(m,60)=1,2,3\\
    1&\text{when }(m,60)=4,5,6,10,15\\
    2&\text{when }(m,60)=12,20,30\\
    3&\text{when }(m,60)=60.
  \end{cases}
\end{equation*}

\subsection{$S_4\subset S_6$}
\label{sec:s_4subset-s_6}

Since $|S_4|=24$, it seems worth reducing the size of calculations in this case by considering orbits of $S_4$ as outlined in \cref{prop:betterorbits}. We will use $S'=\Stab_{S_4}(d_iS_4)$.

For $i=2,5$ the stabilizer is $S_3$. The orbits of $S_4$ under the adjoint action of $S_3$ are obtained by subdividing the well-known conjugacy classes of $S_4$ according to the placement of the letter $4$ in the respective cycle structure. Trusting details to the reader, we state:
\begin{align*}
  Q(d_i)=&()+3(1 2)+3(1 4)\\
       &+2(1 2 3)+6(1 2 4)\\
       &+3(1 2)(3 4)\\
       &+6(1 2 3 4).
\end{align*}
From this we obtain
\begin{align*}
  T((4 5))=(4 5)Q((4 5))=&(45)+3(12)(45)+3(154)\\
                            &+2(123)(45)+6(1254)\\
                            &+3(12)(354)\\
                            &+6(12354)
\end{align*}
and
\begin{align*}
  T((456))=(456)Q((456))=&(456)+3(12)(456)+3(1564)\\
                       &+2(123)(456)+6(12564)\\
                       &+3(12)(3564)\\
                       &+6(123564).
\end{align*}
Thus (omitting the neutral element and writing $3:=3()\in \CC \ol S$ etc.)
{\allowdisplaybreaks
\begin{align*}
 \ol\mu_2((45))&=\frac16(1+3+2(123))=\frac13(2+(123)),\\
 \ol\mu_3((45))&=\frac16(3+3(12))=\frac12(1+(12)),\\
 \ol\mu_4((45))&=\frac16(1+3+2(123)+6)=\frac13(5+(123)),\\
 \ol\mu_5((45))&=1,\\
 \ol\mu_6((45))&=\frac 16(1+3+3+2+3)=2\\
 \ol\mu_{10}((45))&=\frac 16(1+3+2(123)+6)=\frac13(5+(123))=\ol\mu_4((45)),\\
 \ol\mu_{12}((45))&=\frac16(1+3+3+2+6+3)=3,\\
 \ol\mu_{15}((45))&=\frac16(3+3(12)+6)=\frac12(3+(12)),\\
 \ol\mu_{30}((45))&=\frac16(1+3+3+2+3+6)=3=\ol\mu_{12}((45))\\
\ol\mu_{20}((45))&=\frac16(1+3+2(123)+6+6)=\frac13(8+(123))
 \ol\mu_{60}((45))&=4,
% \end{align*}
% \begin{align*}
 \ol\mu_2((456))&=0,\\
 \ol\mu_3((456))&=\frac16(1+3(12)+2)=\frac12(1+(12)),\\
 \ol\mu_4((456))&=\frac16(3+3)=1,\\
 \ol\mu_5((456))&=1,\\
 \ol\mu_6((456))&=\frac16(1+3+2+6)=2,\\
 \ol\mu_{10}((456))&=1,\\
 \ol\mu_{12}((456))&=\frac16(1+3+3+2+3+6)=3,\\
 \ol\mu_{15}((456))&=\frac16(1+3(12)+2+6)=\frac12(3+(12)),\\
\ol\mu_{20}((456))&=\frac16(3+6+3)=2\\
 \ol\mu_{30}((456))&=\frac16(1+3+2+6+6)=3,\\
 \ol\mu_{60}((456))&=4.
\end{align*}}

For $d_7$ the calculations are (even) more tedious; we now need the $S_2$-orbits of $S_4$, that is, the subdivision of the conjugacy classes of $S_4$ according to the placement of the letters $3,4$ in the cycle structure. Thus
\begin{align*}
  Q((35)(46))=&()\\
              &+(12)+2(13)+2(14)+(34)\\
              &+2(123)+2(124)+2(134)+2(143)\\
              &+(12)(34)+2(13)(24)\\
              &+2(1234)+2(1243)+2(1324)
\end{align*}
and
\begin{align*}
  T((35)(46))=&(35)(46)\\
              &+(12)(35)(46)+2(153)(46)+2(164)(35)+(3645)\\
              &+2(1253)(46)+2(1264)(35)+2(15364)+2(16453)\\
              &+(12)(3645)+2(153)(264)\\
              &+2(125364)+2(126453)+2(153264),
\end{align*}
giving
\begin{align*}
  \ol\mu_2((35)(46))&=\ol\mu_3((35)(46))=1,\\
  \ol\mu_4((35)(46))&=4,\\
  \ol\mu_5((35)(46))&=2,\\
  \ol\mu_6((35)(46))&=7,\\
  \ol\mu_{10}((35)(46))&=3,\\
  \ol\mu_{12}((35)(46))&=10,\\
  \ol\mu_{15}((35)(46))&=3,\\
  \ol\mu_{20}((35)(46))&=6,\\
  \ol\mu_{30}((35)(46))&=9,\\
  \ol\mu_{60}((35)(46))&=12.
\end{align*}
In particular, the indicators of the two simples in the image of $\Indtobim_{(35)(46)}$ are identical, while for the other cases we have to distinguish between the three irreducible representations of $S_3$, to wit the trivial representation $W_0$, the sign representation $W_1$ and the two-dimensional irreducible $W_2$. We obtain:
\begin{align*}
  \begin{array}[c]{|c|c||c|c|c|c|c|c|c|c|c|c|c|}
    \hline
    \multicolumn{2}{|c||}{\text{object}}&\multicolumn{11}{|c|}{\nu_m\text{ with }m=}\\
    \hline
    d_i&W_j&2&3&4&5&6&10&12&15&20&30&60\\
    \hline\hline
    \multirow3*{(45)}&W_0&1&1&2&1&2&2&3&2&3&3&4\\
                     &W_1&1&0&2&1&2&2&3&1&3&3&4\\
                     &W_2&1&1&3&2&4&3&6&3&5&6&8\\
    \hline
    \multirow3*{(456)}&W_0&0&1&1&1&2&1&3&2&2&3&4\\
                      &W_1&0&0&1&1&2&1&3&1&2&3&4\\
                      &W_2&0&1&2&2&4&2&6&3&4&6&8\\
    \hline
    (35)(46)&\text{any}&1&1&4&2&7&3&10&3&6&9&12\\
    \hline
  \end{array}
\end{align*}

\subsection{$S_5\subset S_7$}

If we want to deal with the representations associated to $d_7=(46)(57)$ as in the preceding example, we calculate with a sum $Q((46)(57))$ with as many terms as there are orbits in $S_5$ of the adjoint action of $S_3$. One can check that there are $28$ orbits. But we can reduce the task considerably (if not quite by half) by extending the stabilizer to a larger group $S'$ as indicated in \cref{prop:betterorbits}. As the element $(45)(67)$ commutes with $d_7$ and $\Stab_{S_5}(d_7S_5)$, and normalizes $S_5$, we can choose $S'=S_3\cdot\langle(45)(67)\rangle$. Thus we get
\begin{align*}
  Q((46)(57))=&()+3(12)+6(14)+(45)\\
             &+2(123)+12(124)+6(145)\\
             &+6(12)(34)+3(12)(45)+6(14)(25)\\
             &+12(1234)+12(1245)+6(1425)\\
             &+6(12)(345)+12(14)(235)+2(45)(123)\\
             &+12(12345)+12(12435)
\end{align*}
with ``only'' $18$ terms. We calculate
\begin{align*}
  T((46)(57))=&(46)(57)+3(12)(46)(57)+6(164)(57)+(4756)\\
     &+2(123)(46)(57)+12(1264)(57)+6(16475)\\
     &+6(12)(364)(57)+3(12)(4756)+6(164)(275)\\
     &+12(12364)(57)+12(126475)+6(164275)\\
     &+6(12)(36475)+12(164)(2375)+2(4756)(123)\\
     &+12(1236475)+12(1264375).
\end{align*}
From here, we can go through all the divisors $m$ of the exponent $420$ of $S_7$ to obtain the elements $\ol\mu_m$ and the indicators for the three irreducible representations of $S_3$. The table \cref{fig:4657} calculates $\ol\mu_m$ in two stages, giving first an ``unsimplified'' version of $\pi(T^{[m]})$ in an attempt to hint at how this intermediate result can really be read off quite directly from the expression for $T$ obtained above.
\begin{figure}\caption{Indicator calculations on $\operatorname{Im}(\Indtobim_{(46)(57)})\subset\HM S_7..S_5.S_5.$}\label{fig:4657}
\begin{sideways}
\begin{minipage}{\textheight}
\begin{align*}
  \begin{array}[c]{|r|c|c|c|c|c|}
    \hline
    {m}&{\pi(T((46)(57))^{[m]})}
&{\ol\mu_m((46)(57))}&\multicolumn{3}{|c|}{\nu_m(\Indtobim_{(46)(57)}(W_j))}\\
&&&W_0&W_1&W_2\\
    \hline\hline
    2&1+3+2(123)&\frac13(2+(123))&1&1&1\\
\hline
    3&6&1&1&1&2\\
\hline
    4&1+3+1+2(123)+12+3+2(123)&\frac13(10+2(123))&4&4&6\\
\hline
    5&6+6(12)&1+(12)&2&0&2\\
\hline
    6&1+3+6+2+6+6+12+6&7&7&7&14\\
\hline
    7&12+12&4&4&4&8\\
\hline
    10&1+3+2(123)+6+12+6&\frac13(14+(123))&5&5&9\\
\hline
    12&1+3+6+1+2+12+6+3+6+12+6+12+2&12&12&12&24\\
\hline    
    14&1+3+2(123)+12+12&\frac13(14+(123))&5&5&9\\
\hline
    15&6+6+6(12)&2+(12)&3&1&4\\
\hline
    20&1+3+1+2(123)+12+6+3+12+6+2(123)&\frac13(22+2(123))&8&8&14\\
\hline
    21&6+12+12&5&5&5&10\\
\hline
    28&1+3+1+12+3+2(123)+12+12&\frac13(22+2(123))&8&8&14\\
\hline
    30&1+3+6+2+6+6+6+12+12+6+6&11&11&11&22\\
\hline
    35&6+6(12)+12+12&5+(12)&6&4&10\\
\hline
    42&1+3+6+2+6+6+12+6+12+12&11&11&11&22\\
\hline
    60&1+3+6+1+2+12+6+6+3+6+12+12+6+6+12+2&16&16&16&32\\
\hline
    70&1+3+2(123)+6+12+6+12+12&\frac13(26+(123))&9&9&17\\
\hline
    84&1+3+6+1+2+12+6+3+6+12+6+12+2+12+12&16&16&16&32\\
\hline
    105&6+6+6(12)+12+12&6+(12)&7&5&12\\
\hline
    140&1+3+1+2(123)+12+6+3+12+6+2(123)+12+12&\frac13(34+2(123))&12&12&22\\
\hline
    210&1+3+6+2+6+6+6+12+12+6+6+12+12&15&15&15&30\\
\hline
    420&&20&20&20&40\\
    \hline
  \end{array}
\end{align*}
\end{minipage}
\end{sideways}
\end{figure}

For good measure, we shall also finish the calculations for $d_2=(56)$ and $d_5=(567).$ In each case $\Stab_{S_5}(d_iS_5)=S_4$, and
\begin{align*}
  Q(d_i)=&()+6(12)+4(15)\\
         &+8(123)+12(125)\\
         &+3(12)(34)+12(12)(35)\\
         &+6(1234)+24(1235)\\
         &+8(123)(45)+12(125)(34)\\
         &+24(12345),
\end{align*}
thus 
\begin{align*}
  T((56))=&(56)+6(12)(56)+4(165)\\
          &+8(123)(56)+12(1265)\\
          &+3(12)(34)(56)+12(12)(365)\\
          &+6(1234)(56)+24(12365)\\
          &+8(123)(465)+12(1265)(34)\\
          &+24(123465)
\end{align*}
and
\begin{align*}
  T((567))=&(567)+6(12)(567)+4(1675)\\
           &+8(123)(567)+12(12675)\\
           &+3(12)(34)(567)+12(12)(3675)\\
           &+6(1234)(567)+24(123675)\\
           &+8(123)(4675)+12(12675)(34)\\
           &+24(1234675).
\end{align*}
Thus we obtain the elements $\ol\mu_m((56))$ and $\ol\mu_m((567))$ listed in \cref{fig:mubars}.
\begin{figure}\caption{$\ol\mu_m((56)),\ol\mu_m((567))\in \CC S_4$ for indicators in $\HM S_7..S_5.S_5.$}\label{fig:mubars}
\begin{align*}
  \begin{array}[c]{|r|c|c|}
    \hline
    m&\ol\mu_m((56))&\ol\mu_m((567))\\
    \hline
    2&\frac1{14}(5+4(123)+3(12)(34))&0\\
\hline
    3&\frac12(1+(12))&\frac18(3+2(12)+(12)(34)+2(1234))\\
\hline
    4&\frac13(5+(123))&\frac16(4+2(123))\\
    \hline
    5&1&\frac12(1+(12))\\
    \hline
    6&\frac14(11+(12)(34))&\frac14(7+(12)(34))\\
    \hline
    7&0&1\\
    \hline
    10&\frac1{12}(17+4(123)+3(12)(34))&1\\
    \hline
    12&4&3\\
    \hline
    14&\frac1{12}(5+4(123)+3(12)(34))&1\\
    \hline
    15&\frac12(3+(12))&\frac18(7+6(12)+(12)(34)+2(1234))\\
    \hline
   20&\frac13(8+(123))&\frac16(10+2(123))\\
\hline 21&\frac12(1+(12))&\frac18(11+2(12)+(12)(34)+2(1234))\\
    \hline
    28&\frac16(10+2(123))&\frac13(5+(123))\\
    \hline
    30&\frac14(15+(12)(34))&\frac14(11+(12)(34))\\
    \hline
    35&1&\frac12(3+(12))\\
    \hline
    42&\frac14(11+(12)(34))&\frac14(11+(12)(34)\\
    \hline
    60&5&4\\
    \hline
    70&\frac1{12}(17+4(123)+3(12)(34))&2\\
    \hline
    84&4&4\\
    \hline
    105&\frac12(3+(12))&\frac18(15+6(12)+(12)(34)+2(1234))\\
    \hline
    210&\frac14(15+(12)(34))&\frac14(15+(12)(34))\\
    \hline
    420&5&5\\
    \hline
  \end{array}
\end{align*}
\end{figure}
From the information in \cref{fig:mubars} and the following character table of $S_4$
\begin{align*}
  \begin{array}[c]{|r|r|r|r|r|r|}
    \hline
    &()&(12)&(123)&(12)(34)&(1234)\\
\hline
    \eta_0&1&1&1&1&1\\
    \eta_1&1&-1&1&1&-1\\
    \eta_2&2&0&-1&2&0\\
    \eta_3&3&1&0&-1&-1\\
    \eta_4&3&-1&0&-1&1\\
    \hline
  \end{array}
\end{align*}
one can then calculate all the indicator values for the simples in the images of $\Indtobim_{(56)}$ and $\Indtobim_{(567)}$; see \cref{fig:56567}.

\begin{figure}\caption{Indicators on $\operatorname{Im}(\Indtobim_{(56)}),\operatorname{Im}(\Indtobim_{(567)})\subset\HM S_7..S_5.S_5.$}\label{fig:56567}
\begin{align*}
  \begin{array}[c]{|r||c|c|c|c|c||c|c|c|c|c|}
    \hline
    m&\multicolumn{5}{|c||}{\nu_m(\Indtobim_{(56)}(W_i))}&\multicolumn{5}{|c|}{\nu_m(\Indtobim_{(567)}(W_i))}\\
     &W_0&W_1&W_2&W_3&W_4&W_0&W_1&W_2&W_3&W_4\\
    \hline
    2&1&1&1&1&1&0&0&0&0&0\\
\hline
    3&1&0&1&2&1&1&0&1&1&1\\
\hline
    4&2&2&3&5&5&1&1&1&2&2\\
    \hline
    5&1&1&2&3&3&1&0&1&2&1\\
    \hline
    6&3&3&6&8&8&2&2&4&5&5\\
    \hline
    7&0&0&0&0&0&1&1&2&3&3\\
    \hline
    10&2&2&3&4&4&1&1&2&3&3\\
    \hline
    12&4&4&8&12&12&3&3&6&9&9\\
    \hline
    14&1&1&1&1&1&1&1&2&3&3\\
    \hline
    15&2&1&3&5&4&2&0&2&3&2\\
    \hline
   20&3&3&5&8&8&2&2&3&5&5\\
\hline 21&1&0&1&2&1&2&1&3&4&4\\
    \hline
    28&2&2&3&5&5&2&2&3&5&5\\
    \hline
    30&4&4&8&11&11&3&3&6&8&8\\
    \hline
    35&1&1&2&3&3&2&1&3&5&4\\
    \hline
    42&3&3&6&8&8&3&3&6&8&8\\
    \hline
    60&5&5&10&15&14&4&4&8&12&12\\
    \hline
    70&2&2&3&4&4&2&2&4&6&6\\
    \hline
    84&4&4&8&12&12&4&4&8&12&12\\
    \hline
    105&2&1&3&5&4&3&1&4&6&5\\
    \hline
    210&4&4&8&11&11&4&4&8&11&11\\
    \hline
    420&5&5&10&15&15&5&5&10&15&15\\
    \hline
  \end{array}
\end{align*}
\end{figure}

The piece of GAP \cite{GAP4} code in~\cref{fig:GAP} can be used to calculate the higher indicators for objects in $\HM G..H.H.$ for any finite group $G$ and subgroup $H$ available to GAP. It uses the simple but inefficient formula \eqref{eq:14}. Moreover it is written in the most straightforward manner, makes hardly any attempt to reduce the load of calculations, and blindly repeats the same steps several times instead. We do not pursue for the moment the quest to write better code (storing intermediate results such as the elements $\mu_m$ instead of recalculating them for each representation), nor the task to make use of the improved formula in \cref{prop:betterorbits} to speed up matters. The clumsy code is sufficient to do any of the calculations done above ``by hand'' again in seconds. Thus it \emph{could} have been used to verify these results \emph{if} the author \emph{had} had any reason to mistrust his capability to perform flawless computations. Also, \emph{if} the original calculations \emph{had} contained errors, the GAP code \emph{could} have been used to track those down and possibly correct them.

As it stands, the code was also sufficiently efficient to check that the inclusions $S_6\subset S_8$ as well as $S_7\subset S_9$ continue to produce only nonnegative indicator values.
\lstset{language=GAP}
\begin{figure}\caption{GAP code to compute indicators in $\HM G..H.H.$}\label{fig:GAP}
\begin{lstlisting}
IndicatorForOneRep:=function(m,G,H,d,S,eta)
  local h,sum;
  sum:=0;
  for h in H do
    if (d*h)^m in S 
      then sum:=sum+((d*h)^(-m))^eta;
    fi;
  od;
  return(sum/Size(S));
end;

IndicatorsForDoubleCoset:=function(G,H,d)
  local S,eta,irreps,m;
  S:=Intersection(H,H^(d^(-1)));
  irreps:=Irr(S);
  for m in DivisorsInt(Exponent(G)) do
    Print(m,":");
    for eta in irreps do
      Print(IndicatorForOneRep(m,G,H,d,S,eta),",");
    od;
    Print("\n");
  od;
end;

\end{lstlisting}
\end{figure}

\clearpage
\printbibliography
\end{document}